\documentclass[a4paper]{amsart}
\usepackage{amsmath, amsfonts,amsthm,amssymb,amscd, verbatim,graphicx,color,multirow,booktabs, caption, mathdots,bm,chngcntr,adjustbox,longtable,mathtools,microtype,booktabs}

\usepackage[top = 78pt, bottom = 72pt, inner=108pt, outer=108pt]{geometry}
\usepackage[dvipsnames, table]{xcolor}
\usepackage{hyperref}
\hypersetup{colorlinks=true,linkcolor=Maroon, citecolor=OliveGreen}

\usepackage[capitalise,noabbrev]{cleveref}
\usepackage[shortlabels]{enumitem}

\renewcommand{\P}{\Gamma_{\mathrm{P}}}
\newcommand{\EP}{\Gamma_{\mathrm{EP}}}
\newcommand{\CP}{\mathcal{C}\Gamma_{\mathrm{P}}}
\newcommand{\CEP}{\mathcal{C}\Gamma_{\mathrm{EP}}}

\newcommand{\w}{\mathbf{w}}
\renewcommand{\o}{\mathbf{o}}

\title{Elementary proofs of the diameter bounds for the power graphs}

\author[M.~Barbieri]{Marco Barbieri}
\address{Dipartimento di Matematica ``Felice Casorati", University of Pavia, Via Ferrata 5, 27100 Pavia, Italy.} 
\email{marco.barbieri07@universitadipavia.it}

\author[K.~Rekv\'enyi]{Kamilla Rekv\'enyi} 
\address{Department of Mathematics, University of Manchester,  M13 9PL Manchester, UK. Also affiliated with: Heilbronn Institute for Mathematical Research, BS8 1UG Bristol, UK.}
\email{kamilla.rekvenyi@manchester.ac.uk}


\newtheorem{theorem}{Theorem}[]
\newtheorem{lemma}[theorem]{Lemma}
\newtheorem{corollary}[theorem]{Corollary}

\theoremstyle{definition}

\newtheorem{remark}[theorem]{Remark}

\begin{document}
\begin{abstract}
    We give a simplified version of the proofs that, outside of their isolated vertices, the complement of the enhanced power graph and of the power graph are connected of diameter at most $3$.
\end{abstract}
\maketitle

\section{Introduction}
Graphs arising from groups have become a fashionable topic in the last fifteen years: \cite{ALCO_2023__6_5_1395_0,SurveyNilpotency,LucchiniDaniele,Rekvenyi}, just to make a few examples.
In this note, we give a simplified proof of the following results.

\begin{theorem}\label{thm:main}
    Let $G$ be a finite group. Then, outside of its isolated vertices, the complement of the enhanced power graph of $G$ is connected, and its diameter is at most $3$.
\end{theorem}

\begin{corollary}\label{corollary}
    Let $G$ be a finite group. Then, outside of its isolated vertices, the complement of the power graph of $G$ is connected, and its diameter is at most $3$.
\end{corollary}

The problem of the connectedness of the complement of the enhanced power graph has been posed in \cite[Question~20]{MR4346241}. We are aware of two distinct solutions of this problem: \cite[Proposition~3.2]{doi:10.1080/00927870701302081}\footnote{We would also like to point out that \cite{doi:10.1080/00927870701302081} predates \cite{MR4346241}, but their answer has remained unnoticed because their name for the complement of the enhanced power graph is different.} and \cite[Theorem~1.6]{MDW}, both also showing that the unique connected component of such graph has diameter at most $3$. Our proof consists in identifying the largest independent set of in the complement of an enhanced power graph, and in exploiting its properties to show connectedness and bounded diameter.

The connectedness of the complement of the power graph has been proved in \cite[Theorem~9.9]{MR4346241}, while \cite[Question~19]{MR4346241} asks whether or not its diameter is bounded, and an answer in the affirmative has been given in \cite[Theorem~1.5]{MDW}. The novelty of our proof of \cref{corollary} consists of the fact of being an easy consequence of \cref{thm:main}.

We conclude this introduction by pointing out that both bounds in \cref{thm:main,corollary} are sharp. Indeed, for two distinct primes $p$ and $r$, the complement of the power graph and of the enhanced power graph of $C_{pr} \times C_{p^2}$ have diameter $3$ (see~\cite[Lemma~2.6 and Lemma~3.3]{MDW}).

\section{Proof of \cref{thm:main}}
Let $G$ be a finite group. The \emph{enhanced power graph of $G$} is defined as the (undirected) graph $\EP(G)$ of vertex-set $G$ where two vertices are declared adjacent if they are contained in a common cyclic subgroup of $G$. 
We are interested in its complement, which we denote by $\CEP(G)$.

\begin{proof}[Proof of \cref{thm:main}]
Let $G$ be a finite group, and let $g\in G$ be an element of maximal order.\footnote{In choosing $g$, we have already used the finiteness of $G$.} Observe that cliques in $\EP(G)$ are in one-to-one correspondence with maximal cyclic subgroups of $G$ (see~\cite[Proposition~2.4]{MR4346241}). It follows that independent sets in $\CEP(G)$ are in one-to-one correspondence with maximal cyclic subgroups of $G$. In particular, $\langle g \rangle$ is an independent set of maximal size in $\CEP(G)$, and hence, all the elements $h\in G$ which are not powers of $G$ are adjacent to $g$.

Let $g^\ell$ be a power of $g$, and suppose that $g^\ell$ is not an isolated vertex (that is, it is not contained in every maximal cyclic subgroup of $G$). Then, there exists a maximal cyclic subgroup $C\le G$ such that $g^\ell \notin C$. If we denote by $h_\ell$ a generator of $C$, $g^\ell$ and $h_\ell$ are adjacent. Therefore, $g^\ell$ is part of the connected component that contains all the $h\in G$ which are not powers of $g$. This proves that $\CEP(G)$, outside of its isolated vertices, has exactly one connected component.

We now focus on the upper bound on the diameter of the connected component of $\CEP(G)$. Observe that, our proof already shows that the diameter is bounded from above by $4$, and the vertices that might meet the maximal distance are distinct powers of $g$. Aiming for a contradiction, we suppose that, for some integers $a$ and $b$, $g^a$ and $g^b$ are at distance $4$. We suppose that
\[ g^a \sim h_a \sim g \sim h_b \sim g^b \]
is a path of minimal distance connecting $g^a$ and $g^b$.\footnote{We point out that, for $i\in \{a,b\}$, $h_i$ is the vertex adjacent to $g^i$ built in the previous paragraph.} By minimality of this path, $h_a$ and $h_b$ are not adjacent. By the definition of $\CEP(G)$, $\langle h_a, h_b \rangle$ is a cyclic subgroup of $G$, and, for $i\in \{a,b\}$, by construction, $h_i$ generates a maximal cyclic subgroup of $G$. Therefore,
\[ \langle h_a, h_b \rangle = \langle h_a \rangle =\langle h_b \rangle \,.\]
It follows that $g^a$ is adjacent to $h_b$, and $g^b$ is also adjacent to $h_a$. Hence, $g^a$ and $g^b$ are at distance $2$, a contradiction. This proves that the diameter of $\CEP(G)$ is at most $3$.
\end{proof}

\section{Proof of \cref{corollary}}
The adjective \emph{enhanced} suggest that there exist an original version of the \emph{power graph}. Indeed, for a finite group $G$, the power graph $\P(G)$ is defined as the graph on the vertex-set $G$ such that two elements are declared adjacent if one is a power of the other. We denote its complement by $\CP(G)$.

\begin{remark}[\cite{CGS}, Theorem~2.12]\label{remark}
    For any prime $p$, and for any positive integer $m$, $\P(C_{p^m})$ is complete.\footnote{The converse is also true: if $\P(G)$ is complete, then $G$ is a cyclic $p$-group.} Indeed, this is follows from the fact that the series
    \[ C_{p^m} \ge C_{p^m}^{\phantom{p^m}p} \ge \dots \ge C_{p^m}^{\phantom{p^m}p^{m-1}} \ge 1 \,. \]
    contains all the subgroups of $C_{p^m}$. Hence, $\CP(C_{p^m})$ contains no edges.
\end{remark}

\begin{proof}[Proof of \cref{corollary}]
By construction, $\P(G)$ is a subgraph of $\EP(G)$, and hence the complement of the power graph $\CP(G)$ contains $\CEP(G)$ as a subgraph. We need to focus our attention on those isolated vertices of $\CEP(G)$ that have neighbours in $\CP(G)$. In the proof of \cref{thm:main}, we have shown that, all isolated vertices in $\CEP(G)$ are powers of an element of maximal order $g$.

Aiming for a contradiction, we suppose that there is an integer $\ell$ such that $g^\ell$ is contained in a connected component of $\CP(G)$ that does not contain $g$. It follows that $g^\ell$ is not adjacent to any $h \in G - \langle g \rangle$. In particular, $g^\ell$ is a power of every element of the group that is not a power of $g$. If two distinct primes divide the order of $G$, say $p$ and $r$, then $g^\ell$ lies both in a Sylow $p$-subgroup and a Sylow $r$-subgroup, because $g^\ell$ is a power of both a $p$-element and an $r$-element. The only element with this property is the identity, and thus $g^\ell=1$ is an isolated vertex. Therefore, $G$ is a $p$-group, for a suitable prime $p$, and the neighbourhood of $g^\ell$ is contained in $\langle g \rangle$. Since $g$ is a $p$-element, by \cref{remark}, $\CP(\langle g \rangle)$ has no edges. Hence, $g^\ell$ is an isolated vertex, a contradiction. This shows that $\CP(G)$, outside of its isolated vertices, has a single connected component.

We now shift our focus on the diameter of $\CP(G)$. Since $\CP(G)$ is a subgraph of $\CEP(G)$, \cref{thm:main} implies that, if two vertices are at distance $4$ in $\CP(G)$, then they are distinct powers of $g$. Aiming for a contradiction, suppose that, for two suitable integers $a$ and $b$, $g^a$ and $g^b$ are at distance $4$ in $\CP(G)$. Let
\[ g^a \sim x \sim y \sim z \sim g^b\]
be a path of minimal distance connecting $g^a$ and $g^b$. By definition of $\CP(G)$ and by the minimality of the path, we have that
\[ g^a \notin \langle x \rangle \,, \quad g^b \notin \langle z \rangle\,, \quad x \in \langle z  \rangle \,.\footnote{It might happen that $z \in \langle x \rangle$ is true instead: in this case, we just swap $a$ with $b$ and $x$ with $z$.}\]
We obtain that $\langle x \rangle$ is a subgroup of $\langle z \rangle$. Hence, $g^b$ cannot be an element of $\langle x \rangle$, and $g^b$ and $x$ are adjacent, against the minimality of the chosen path.\footnote{We observe that this proof actually never uses the assumption that $g^a$ and $g^b$ are powers of $g$.} This final contradiction completes our proof and our note.
\end{proof}

\bibliographystyle{plain}
\bibliography{refs.bib}
\end{document}